\documentclass[11pt, oneside]{article}   

\usepackage{amsmath}
\usepackage{amsfonts}
\usepackage{amsthm}
\usepackage{amssymb}
\usepackage{mathrsfs}
\usepackage{cite}
\usepackage{tikz}
\usepackage{verbatim}
\usepackage{amsmath,amsfonts,amsthm,amssymb}

\newcommand{\vect}[1]{\overset{\rightharpoonup}{#1}}

\newif\ifshowvc
\showvctrue
\showvcfalse  
\makeatletter
\def\blfootnote{\xdef\@thefnmark{}\@footnotetext}
\makeatother
\ifshowvc
\input{vc}
\fi

\theoremstyle{plain}
\newtheorem{theorem}{Theorem}[section]

\newtheorem{lemma}[theorem]{Lemma}

\theoremstyle{definition}
\newtheorem{definition}[theorem]{Definition}

\theoremstyle{remark}

\newcommand{\beq}{\begin{align*}}
\newcommand{\eeq}{\end{align*}}

\newcommand{\ppp}{\[\begin{aligned}}
\newcommand{\ooo}{\end{aligned}\]}

\begin{document} 

\title{On Taylor series of zeros with general base function}

\author{Mario DeFranco}

\maketitle

\abstract{We prove a formula for the Taylor series coefficients of a zero of the sum of a complex-exponent polynomial and a base function which is a general holomorphic function with a simple zero. Such a Taylor series is more general than a Puiseux series. We prove an integrality result about these coefficients which implies and generalizes the integrality result of Sturmfels (``Solving algebraic equations in terms of $\mathcal{A}$-hypergeometric series". Discrete Math. 210 (2000) pp. 171-181). We also prove a transformation rule for a special case of these Taylor series.} 

\section{Introduction}
Determining the zeros of a polynomial is a fundamental objective throughout pure and applied mathematics. By the work of Abel \cite{Abel}, Ruffini \cite{Ruffini}, and Galois \cite{Galois}, there is no general solution by radicals for polynomials of degree five or more. Nevertheless, formal series expansions may provide a unified way to consider zeros of polynomials of any degree.

Many authors have studied such formal series. Birkeland \cite{Birkeland} obtained formulas for Taylor series coefficients using Lagrange inversion. Mayr \cite{Mayr} also studied these series using a system of differential equations. Sturmfels \cite{Sturmfels2000} obtained the same formulas using GKZ-systems (Gelfand-Kapranov-Zelevinsky) \cite{Gelfand89}, \cite{Gelfand90}. Herrera \cite{Herrera} used reversion of Taylor series. 

In this paper, using Taylor series directly, we prove a formula (Theorem \ref{t phi F}) for the Taylor series coefficients of a zero of a function that may be a polynomial or, more generally, a sum of a complex-exponent polynomial and a holomorphic function (the base function). We have defined a complex-exponent polynomial and the term ``base function" in \cite{DeFranco} (also defined below in Definition \ref{complex exponent poly}), and our proofs use similar techniques to those of \cite{DeFranco}, including the use of Theorem \ref{t deriv set} here. 
We next explain more about the set-up and relation to previous work. 

First we recall complex-exponent polynomials. For two complex numbers $z \neq 0$ and $\gamma$, we define complex exponentiation $z^\gamma$ in the conventional way by taking $z$ to be an element of the Riemann surface $L$ for the logarithm. This surface $L$ is parametrized by 
\[
L=\{(r,\theta,n)\colon r \in \mathbb{R}^+, \theta \in (-\pi, \pi], n \in \mathbb{Z}\}.
\]
Then for $z \in L$ corresponding to $(r,\theta,n)$, define $z^\gamma$ by 
\[
z^\gamma=e^{\gamma \ln(r)+i\gamma \theta+ 2 \pi i n \gamma} \in \mathbb{C}.
\]

\begin{definition} \label{complex exponent poly}
For an integer $d \geq 1$, let $\vect{a}$ and $\vect{\gamma}$ be two $d$-tuples of complex numbers 
\begin{align*}
\vect{a} = (a_1, \ldots, a_d) \\ 
\vect{\gamma} = (\gamma_1, \ldots, \gamma_d).
\end{align*}
Define a complex-exponent polynomial $p(z;\vect{a},\vect{\gamma})$ to be a function of the form
\[
p\colon L \rightarrow \mathbb{C}
\]
\[
z \mapsto \sum_{k=1}^d a_k z^{\gamma_k}
\]
which we abbreviate as $p(z)$.
\end{definition} 
Let $g(z) \colon \mathbb{C} \rightarrow \mathbb{C}$ be a holomorphic function.  
We assume that $g(z)$ has a simple zero $\alpha \neq 0$, and we write its expansion about $\alpha$ as  
\begin{equation} \label{general base}
g(z) = \sum_{k=1}^\infty c_k (z-\alpha)^k
\end{equation}
for some $c_k \in \mathbb{C}$ with $c_1 \neq 0$. We call $g(z)$ the ``base function". We fix an element of $L$ that corresponds to $\alpha$ and also denote it by $\alpha$. Then $g(z)$ also determines a function from a neighborhood $V$ of $\alpha$ in $L$ to $\mathbb{C}$, via the mapping 
\[
z \mapsto z^1 \in \mathbb{C} \text{ for } z \in V.
\]

Now fix an integer $d\geq 1$ and $\vect{\gamma} \in \mathbb{C}^d$. For an $\vect{a} \in \mathbb{C}^d$, define the function 
\begin{align*}
&f(z; \vect{a}, \vect{\gamma}, g) \colon L \rightarrow \mathbb{C} \\ 
&z \mapsto  g(z)+\sum_{i=1}^d a_i z^{\gamma_i}
\end{align*} 

Assume, for some neighborhood $U \subset \mathbb{C}^d$ of the origin $\vect{0}$, that there exists a smooth function 
\begin{align*}
&\phi(\vect{a}; \vect{\gamma}, g, \alpha) \colon U \rightarrow L \\ 
&\vect{a} \mapsto  \phi(\vect{a}; \vect{\gamma}, g, \alpha)
\end{align*} 
that satisfies for all $\vect{a} \in U$
\begin{align}
&f((\phi(\vect{a}; \vect{\gamma}, g, \alpha);  \vect{a}, \vect{\gamma}, g)) = 0 \label{zero}\\ 
& \phi(0; \vect{\gamma}, g, \alpha) = \alpha \label{phi zero}.
\end{align} 
Equations \eqref{zero} and \eqref{phi zero} are sufficient to determine the Taylor series coefficients of $\phi(\vect{a}; \vect{\gamma}, g, \alpha)$ about $\vect{0}$ in the variables $\vect{a}$.
 
Let $\vect{n}$ denote a $d$-tuple of non-negative integers 
\[
\vect{n} = (n_1, \ldots, n_d)
\]
and denote  
\[
\Sigma \vect{n} = \sum_{i=1}^d n_i.
\]
Let $\partial_{\vect{n}}$ denote the partial derivative operator
\[
\partial_{\vect{n}} = \prod_{i=1}^d (\frac{\partial}{\partial a_i})^{n_i},
\]
and for any function 
\begin{equation}\label{psi form}
\psi \colon : U \rightarrow \mathbb{C}
\end{equation}
denote
\[
\partial(\psi, \vect{n})=\partial_{\vect{n}} \psi(\vect{a}) |_{\vect{0}}.
\]

The Taylor series coefficient formula of Theorem 1.3, \cite{DeFranco} (stated as Theorem \ref{t main} here) is obtained using a base function $g(z)$ of the form 
\begin{equation} \label{2 term base}
g(z) = 1 + b z^\beta
\end{equation}
for non-zero $b, \beta \in \mathbb{C}$. As shown in \cite{DeFranco}, Birkeland \cite{Birkeland} and Sturmfels \cite{Sturmfels2000} essentially find the Taylor series coefficients using the base function \eqref{2 term base} with positive integer $\beta$ and with $\gamma_i$ distinct integers not equal to $\beta$, though they do not use that terminology or the methods of Taylor series. Their formulas 
may be expressed in terms of falling factorials. We generalize these formulas in \cite{DeFranco} by allowing $\gamma_i$ and $\beta$ to be arbitrary complex numbers $(\beta \neq 0)$ and show that the factorization is preserved using falling factorials (see Theorem \ref{t main} here). 

In this paper, we consider the base function \eqref{general base} and prove in Theorem \ref{t phi F} that $\partial(\phi, \vect{n})$ is a polynomial in $c_1^{-1}, c_i, 2 \leq i \leq \Sigma \vect{n}$, and that the coefficients of this polynomial factorize as well, using similar falling factorials. 

Now, the base function \eqref{2 term base} is a special case of \eqref{general base} obtained via 
\[
c_k = b {\beta \choose k} \alpha^{\beta-k}.
\]
Using this fact combined with Theorem \ref{t Taylor int} gives another proof of Theorem 4.1 in Sturmfels \cite{Sturmfels2000}. His proof of Theorem 4.1 uses Hensel's Lemma, and our proof of Theorem \ref{t Taylor int} counts terms in a derivative. 

We also note with base function \eqref{2 term base}, Theorem \ref{t phi F} here and  Theorem 1.3 of \cite{DeFranco} give two different formulas for the quantity $\partial(\phi, \vect{n})$. We include in section \ref{further} the objective of finding a direct proof of this identity. 


\section{Definitions and cited theorems}

\subsection{Definitions}

Most of the following definitions occur in \cite{DeFranco}.

First we present notation for multisets and multiset partitions. For an integer $M\geq 0$, we let $[1,M]$ denote 
\[
[1,M] = \{i \in \mathbb{Z} \colon 1 \leq i \leq M\}.
\]

\begin{definition}
For a positive integer $N$, define an ordered multiset $I$ of $[1,d]$ to be an $N$-tuple of integers 
\[
I = (I(1), \ldots, I(N))
\]
with $1 \leq I(i) \leq d$. We say that the order $|I|$ is $N$. We define the multiplicity $\#(n,I)$ of $n$ in $I$ as the number of indices $i$ such that $I(i)=n$. For a multiset $X$ we also denote the multiplicity of $n$ in $X$ as $\#(n,X)$. 

Let $\mathrm{Multiset}(d)$ denote the set of these ordered multisets. 

For a positive integer $k$, define a set partition $s$ of $[1,N]$ with $k$ parts to be a $k$-tuple
\[
s = (s_1, \ldots, s_k)
\]
where $s_i$ are pairwise disjoint non-empty subsets of $[1,N]$, 
\[
\bigcup_{i=1}^k s_i = [1,N],
\]
and 
\[
\min(s_i) < \min(s_j) \text{ for } i <j.
\]
We also write a set $s_i$ as an $m$-tuple
\[
s_i = (s_i(1), \ldots, s_i(m))
\]
where $m=|s_i|$ and 
\[
s_i(j) < s_i(l) \text{ for } j <l.
\]
Let $S(N,k)$ denote the set of such $s$. If $H$ is any finite set of integers, we similarly denote $S(H,k)$ to be the set of all set partitions of $H$ into $k$ non-empty parts.

For a multiset $I$ and a set partition $s \in S(|I|,k)$, define a multiset partition $J$ of $I$ with $k$ parts to be a $k$-tuple
 \[
 J = (J_1, \ldots, J_k)
 \]
 where $J \in \mathrm{Multiset}(d)$ is given by 
 \[
 J_i = (I(s_i(1)), \ldots, I(s_i(m)))
 \]
 where $m = |s_i|$. Thus the multiset partitions of $I$ with $k$ parts are in bijection with the set partitions in $S(|I|, k)$. 
Let $\mathrm{Parts}(I,k)$ denote the set of such multiset partitions $J$.  

Each part $J_i$ of a $J \in \mathrm{Parts}(I,k)$ is thus a multiset. For $J,J' \in  \mathrm{Parts}(I,k)$, we say that $J$ and $J'$ are equivalent if there exists a permutation $\sigma$ of $[1,k]$ such that
\[
\#( l,J_i) = \#(l,J_{\sigma(i)}')
\]
for each $1\leq l \leq d$.

 We let $ I(\hat{h})$ denote the ordered multiset obtained from $I$ by removing the element at the $h$-th index:
  \[
 I(\hat{h}) = (I(1), \ldots, I(h-1), I(h+1), \ldots,I(N)).
 \]
 We use the notation 
 \[
 \sum_{m \in I} \gamma_m = \sum_{i=1}^N \gamma_{I(i)}
 \]
where $N = |I|$.

For any function $\psi(\vect{a})$ of the form \eqref{psi form} and $I \in 
\mathrm{Multiset}(d)$, denote 
\[
\partial(\psi, I) = (\prod_{i=1}^{|I|} \frac{\partial}{\partial a_{I(i)}} ) \psi(\vect{a}) |_{\vect{0}}
\]
\end{definition} 
We also use the falling factorial applied to indeterminates, where ``indeterminate" refers to an arbitrary element of some polynomial ring over $\mathbb{Z}$. 
\begin{definition}
For an integer $k \geq 0$ and an indeterminate $x$, define the falling factorial 
\[
(x)_k = \prod_{i=1}^k(x-i+1).
\]
\end{definition}

Let $\mu$ denote an infinite sequence $(\mu_i)_{i=1}^\infty$ of non-negative integers $\mu_i$ such that $\mu_i=0$ for sufficiently large $i$. For integer $r \geq 1$, let
$C(r)$ denote the set of all such $\mu$ such that 
\[
\sum_{i\geq 1} \mu_i = r.
\]
Thus $C(r)$ is the set of compositions of $r$ with non-negative parts.

\subsection{Cited theorems}

These results are used here in Theorems \ref{t phi F}, \ref{t F prod}, and \ref{t phi transform}.

\begin{lemma} \label{fall identity}
For indeterminates $a$ and $b$ and an integer $n \geq 0$, 
\[
(a+b)_n = \sum_{i=0}^n {n \choose i} (a)_i(b)_{n-i}
\]
and equivalently 
\[
{a+b \choose n} = \sum_{i=0}^n {a \choose i} {b \choose n-i}.
\]
\end{lemma}

\begin{proof} 
This is Lemma 6.1 of \cite{DeFranco}.
\end{proof} 

\begin{lemma} \label{Newton series}
Suppose $F(x)$ is a polynomial of degree $m$. Then
\[
F(x) = \sum_{k=1}^{m+1} \frac{(x-1)_{k-1}}{(k-1)!}\sum_{r=0}^{k-1} (-1)^{k-1-r} {k-1 \choose r}F(r+1).
\]
\end{lemma} 
\begin{proof} 
This is Lemma 2.4 of \cite{DeFranco}.
\end{proof} 

\begin{theorem} \label{t main}
With the above notation and base function \eqref{2 term base}, for $\Sigma \vect{n} \geq 1$,
\[
\partial_{\vect{n}} \phi(\vect{a}) |_{\vect{0}} = -\frac{\alpha^{1 + \sum_{i=1}^d n_i(\gamma_i-1)}}{g'(\alpha)^{\Sigma \vect{n}}} \prod_{i=1}^{\Sigma \vect{n} -1} (-1+i \beta - \sum_{i=1}^d n_i\gamma_i)
\]
where 
\[
g'(\alpha) = b \beta \alpha^{\beta -1}.
\]
\end{theorem}
\begin{proof} 
This is Theorem 1.3 of \cite{DeFranco}.
\end{proof} 

 \begin{theorem} \label{t deriv set}
 Let $R$ be a commutative ring and let $\delta \colon R \rightarrow R$ be a derivation. 
 For an integer $M\geq 0$ and elements $f_A, f_B, f_i \in R, 1 \leq i \leq M$, 
 \begin{equation}\label{deriv set}
\sum_{ w \subset [1,M] }\delta^{|w^c|-1} (f_A^{(1)} \prod_{i \in w^c} f_i ) \delta^{|w|} (f_B \prod_{i \in w} f_i )  = \delta^M (f_A f_B \prod_{i=1}^M f_i)
  \end{equation}
  where $w^c$ denotes the complement of $w$ in $[1,M]$; $\displaystyle   f_A^{(1)}$ denotes $\delta f_A$; and in the case $w^c$ is empty, $\displaystyle   \delta^{-1} f_A^{(1)}$ denotes $f_A$.
 \end{theorem}
\begin{proof} 
This is Theorem 4.3 of \cite{DeFranco}.
\end{proof}

\begin{theorem} \label{t nu}
For integers $1 \leq k \leq N$, an indeterminate $\nu$, and $N$ indeterminates $x_i$, $1 \leq i \leq N$,
\begin{align}
&\frac{1}{(k-1)!}\sum_{r=0}^{k-1} (-1)^{k-1-r}{k-1 \choose r} ((r+1)\nu-1+\sum_{i=1}^N x_i)_{N-1} \label{nu 1}\\
=&\nu^{k-1}\sum_{s \in S(N,k)} \prod_{i=1}^{k} (\nu-1+\sum_{m\in s_i} x_m)_{|s_i|-1}. \label{nu 2}
\end{align}
\end{theorem}
\begin{proof} 
This is Theorem 4.1 of \cite{DeFranco}.
\end{proof}

\section{Taylor coefficient theorem}

\begin{definition}
Let $\alpha$ be the zero of base function \eqref{general base}. For an $x \in \mathbb{C}$ and integers $r \geq 0$ and $a \geq 1$, define
\begin{align*}
&F(x,r,a) =  \frac{-\alpha^{x}}{(a-1)!} \sum_{\mu \in C(r)} (-1)^{\mu_1} {x \choose r-\sum_{i\geq 2} (i-1)\mu_i-(a-1)}  (r+ \sum_{i \geq 2} \mu_i)! \\ 
& \times \frac{  \alpha^{-(r-\sum_{i\geq 2} (i-1)\mu_i-(a-1))}
}{ c_1^{r+1+\sum_{i \geq 2}\mu_2}}\prod_{i\geq 2}\frac{c_i^{\mu_i}}{\mu_i!}
\end{align*}
and also
\[
F(x,-1,a) =0.
\]
\end{definition} 

\begin{theorem}\label{t phi F}
With the above notation and base function \eqref{general base}, and for an $I \in \mathrm{Multiset}(d)$ with $|I|\geq 1$, 
\begin{equation}\label{phi F}
\partial(\phi,I) = F(\sum_{m \in I} \gamma_m, |I|-1, 1). 
\end{equation}
\end{theorem}

\begin{proof}
We use induction on $|I|$. When $I=(i)$, we differentiate equation \eqref{zero} with respect to $a_i$ and evaluate at $\vect{a}=0$ to obtain
\[
0=\alpha^{\gamma_i} + c_1 \partial(\phi,I).
\]
Solving for $\partial(\phi,I)$ proves the base case of $|I|=1$. 

Given an $I$ with $|I|\geq 1$, suppose equation \eqref{phi F} is true for all $I'$ with $|I'|< |I|$. Given any function $\psi(\vect{a})$ of the form $\eqref{psi form}$,
it follows from the definitions that
\begin{align}
\partial(f \circ \psi,I)  &= \sum_{h=1}^{|I|} \sum_{k=1}^{|I|-1} (\gamma_{I(h)})_k\psi(\vect{0})^{\gamma_{I(h)}-k}\sum_{J \in \mathrm{Parts}(I(\hat{h}), k)}\prod_{i=1}^k \partial(\psi,J_i)) \label{psi f} \\ 
&+  \sum_{k=1}^{|I|} g^{(k)}(\psi(\vect{0})) \sum_{J \in \mathrm{Parts}(I, k)}\prod_{i=1}^k \partial(\psi,J_i). \label{psi g}
\end{align}

Setting $\psi$ to be $\phi$, we obtain 

\begin{align}
0 = & \sum_{h=1}^{|I|} \sum_{k=1}^{|I|-1} (\gamma_{I(h)})_k\alpha^{\gamma_{I(h)}-k}\sum_{J \in \mathrm{Parts}(I(\hat{h}), k) }\prod_{i=1}^k \partial(\phi,J_i) \label{phi f}\\ 
&+  \sum_{k=2}^{|I|} k!c_k \sum_{J \in \mathrm{Parts}(I, k)}\prod_{i=1}^k \partial(\phi,J_i) \label{k c} \\ 
&+c_1 \partial(\phi,I) \label{phi I}.
\end{align}
Using the induction hypothesis we may express the sum in the right side of the equation at  line \ref{phi f} as 
\[
 \sum_{h=1}^{|I|} \sum_{k=1}^{|I|-1} (\gamma_{I(h)})_k\alpha^{\gamma_{I(h)}-k}\sum_{J \in \mathrm{Parts}(I(\hat{h}), k) }\prod_{i=1}^k F(\sum_{m \in J_i}x_m, |J_i|-1,1).
\] 
To the sum over $J$ we apply Theorem \ref{t F prod} and obtain 
\[
 \sum_{h=1}^{|I|} \sum_{k=1}^{|I|-1} (\gamma_{I(h)})_k\alpha^{\gamma_{I(h)}-k} F(\sum_{m \in I(\hat{h})}x_m, |I|-2,k).
\]
Substituting in the definition of $F(x,r,a)$ and then using 
\[
(\gamma_{I(h)})_k = (\gamma_{I(h)}) (\gamma_{I(h)}-1)_{k-1}
\]
yields
\begin{align*}
&\sum_{h=1}^{|I|} -\alpha^{\sum_{i \in I} \gamma_i} \gamma_{I(h)}\sum_{k=1}^{|I|-1} \frac{(\gamma_{I(h)}-1)_{k-1}}{(k-1)!}  \sum_{\mu \in C(|I|-2)} (-1)^{\mu_1} {\sum_{i \in I(\hat{h})} \gamma_i  \choose |I|-2-\sum_{i\geq 2} (i-1)\mu_i-(k-1)}  \\
& \times(|I|-2+ \sum_{i \geq 2} \mu_i)! 
 \frac{  \alpha^{-(|I|-1-\sum_{i\geq 2} (i-1)\mu_i))}
}{ c_1^{|I|-1+\sum_{i \geq 2}\mu_2}}\prod_{i\geq 2}\frac{c_i^{\mu_i}}{\mu_i!}  .
\end{align*}
Interchanging the sum over $k$ and $\mu$ and applying Lemma \ref{fall identity} gives 
\begin{align*}
&\sum_{h=1}^{|I|} -\alpha^{\sum_{i \in I} \gamma_i} \gamma_{I(h)}\sum_{\mu \in C(|I|-2)} (-1)^{\mu_1} \frac{(-1+\sum_{i \in I} \gamma_i)_{|I|-2-\sum_{i\geq 2} (i-1)\mu_i}}{(|I|-2-\sum_{i\geq 2} (i-1)\mu_i)!} 
\end{align*}
and summing over $h$ gives 
\begin{align}
-\alpha^{\sum_{i \in I} \gamma_i}\sum_{\mu \in C(|I|-2)} (-1)^{\mu_1} \frac{(\sum_{i \in I} \gamma_i)_{|I|-1-\sum_{i\geq 2} (i-1)\mu_i}}{(|I|-2-\sum_{i\geq 2} (i-1)\mu_i)!}. \label{summed h} 
\end{align}

Now to each term at line \eqref{k c}, we again apply the induction hypothesis and Theorem \ref{t F prod}, and at line \eqref{phi I} we assume the result
\[
 \partial(\phi,I) = F(\sum_{m\in I}x_m, |I|-1,1).
\]
Now combining the result \eqref{summed h} it is sufficient to prove

\begin{align} 
0=&-\sum_{\mu \in C(|I|-2)} \frac{(x)_{|I|-1- \sum_{i \geq 2 }(i-1)\mu_i}}{(|I|-2- \sum_{i \geq 2 }(i-1)\mu_i)!}(-1)^{\mu_1}\frac{(|I|-2+\sum_{i \geq 2} \mu_i)!}{c_1^{|I|-1+\sum_{i \geq 2} \mu_i}}\prod_{i\geq 2} \frac{c_i^{\mu_i}}{\mu_i!} \label{non F}\\ 
&+\sum_{i\geq 2} i!c_iF(\sum_{m \in I} x_m, |I|-1, i) \label{j terms}\\ 
&+c_1F(\sum_{m \in I} x_m,|I|-1,1) \label{c1 F}.
\end{align} 
Take the coefficient of $\prod_{i \geq 2}c_i^{\nu_i}$ from each line. 
Consider the expression
\begin{equation} \label{coef factor}
 \frac{(x)_{|I|-1- \sum_{i \geq 2 }(i-1)\nu_i}}{(|I|-1- \sum_{i \geq 2 }(i-1)\nu_i)!} (-1)^{|I|-1+\sum_{i\geq 2} \nu_i }\frac{(|I|-2+\sum_{i \geq 2} \nu_i)!}{c_1^{|I|-1+\sum_{i \geq 2} \nu_i}\prod_{i \geq 2} \nu_i!}.
\end{equation}
The coefficient from line \eqref{non F} is equal to the expression \eqref{coef factor} times 
\begin{equation}\label{non F part}
|I|-1- \sum_{i \geq 2 }(i-1)\nu_i.
\end{equation}
The coefficient from the $i$-th term at line \eqref{j terms} is equal to the expression \eqref{coef factor} times  
\begin{equation}\label{j term part}
\frac{i!}{(i-1)!}\nu_i.
\end{equation}
The coefficient from line \eqref{c1 F} is equal to the expression \eqref{coef factor} times  
\begin{equation}\label{c1 F part}
-(|I|-1+\sum_{i \geq 2}\nu_i).
\end{equation}
Adding the three expressions yields 0:
\[
|I|-1- \sum_{i \geq 2 }(i-1)\nu_i+ \sum_{i \geq 2} i \nu_i -(|I|-1+\sum_{i \geq 2}\nu_i) = 0. 
\]
This completes the proof. 
\end{proof}

\begin{theorem} \label{t F prod}
For an integers $1 \leq a \leq M$ and indeterminates $x_i, 1 \leq i \leq M$, 
\begin{equation} \label{F prod}
\sum_{s \in S(M, a)} \prod_{i=1}^a F(\sum_{m \in s_i} x_m, |s_i|-1, 1) = F(\sum_{i=1}^{M} x_i, M-1, a).
\end{equation}
\end{theorem}
\begin{proof}
 We use induction on $M$. Equation \eqref{F prod} is true when $M=1$. Assume it is true for all values less than or equal to some $M\geq 1$. We note the equation is true when $a=1$ for any $M$. For $ 1\leq a \leq M$, we have by the induction hypothesis 
\begin{align} 
&\sum_{s \in S(M+1, a+1)} \prod_{i=1}^{a+1} F(\sum_{m \in s_i} x_m, |s_i|-1, 1) \nonumber\\ 
&= \sum_{w\subset [1,M+1], M+1 \in w} F(\sum_{m\in w^c}x_m, |w^c|-1,a)F(\sum_{m\in w}x_m, |w|-1,1)\label{F prod hyp}
\end{align}
where $w$ is the set of the set partition $s$ that contains $M+1$. 
For elements $\nu\in C(M)$ and $\mu \in C(j)$ for some $j$, we say $\mu \leq_1 \nu $ if
\begin{align*}
&\mu_i \leq \nu_i \text{ for } i \geq 2 \\ 
&\mu_1 \leq \nu_1-1. 
\end{align*}
If $\mu \leq_1 \nu$, denote $\mu' \in C(M- 1-j)$ by 
\begin{align*}
&\mu_i' = \nu_i -\mu_i\text{ for } i \geq 2 \\ 
&\mu_1' = \nu_1-1- \mu_1. 
\end{align*}
Fix an element $\nu\in C(M)$. Then the coefficient of 
\begin{equation}\label{c mono}
\prod_{i \geq 2} c_i^{\nu_i}
\end{equation}
in the right side of equation \eqref{F prod hyp} is 
\begin{align}
& \frac{-\alpha^{\sum_{i=1}^{M+1}x_i}}{(a-1)!}\sum_{w \subset [1,M+1], M+1 \in w} \sum_{\mu \leq_1 \nu, \mu\in C(|w|-1)}   \\ 
&\left((-1)^{\mu_1'} {\sum_{m\in w^c} x_m \choose |w^c|-1-\sum_{i\geq 2} (i-1)\mu_i'-(a-1)} (|w^c|-1+ \sum_{i \geq 2} \mu_i')!   \frac{  \alpha^{-(|w^c|-1-\sum_{i\geq 2} (i-1)\mu_i'-(a-1))}
}{ c_1^{|w^c|+\sum_{i \geq 2}\mu_i'} \prod_{i\geq 2}\mu_i'!}\right) \\ 
&\times \left((-1)^{\mu_1} {\sum_{m\in w} x_m \choose |w|-1-\sum_{i\geq 2} (i-1)\mu_i} (|w|-1+ \sum_{i \geq 2} \mu_i)!  \frac{  \alpha^{-(|w|-1-\sum_{i\geq 2} (i-1)\mu_i)}
}{ c_1^{|w|+\sum_{i \geq 2}\mu_i} \prod_{i\geq 2}\mu_i!}\right) \label{b line}.
\end{align}
For an integer $b \geq 0$, the coefficient of $(x_{M+1})_b$ at line \eqref{b line} is 
\begin{align}
& \frac{(-1)^{\nu_1+1}\alpha^{\sum_{i=1}^{M+1}x_i-(M-\sum_{i \geq 2}(i-1)\nu_i-a)}}{c_1^{M+1+\sum_{i\geq 2}\nu_i}} \label{alpha factor}\\ 
&\times \frac{1}{(a-1)!b!}\sum_{u \subset [1,M]} \sum_{\mu \leq_1 \nu, \mu\in C(|u|)}  
 {\sum_{m\in u^c} x_m \choose |u^c|-1-(a-1)-\sum_{i\geq 2} (i-1)\mu_i'} \frac{(|u^c|-1+ \sum_{i \geq 2} \mu_i')! }{ \prod_{i\geq 2}\mu_i'!}  \label{uc factor}\\ 
&\times {\sum_{m\in u} x_m \choose |u|-b-\sum_{i\geq 2} (i-1)\mu_i} \frac{(|u|+ \sum_{i \geq 2} \mu_i)! }{\prod_{i\geq 2}\mu_i!} \label{u factor}
\end{align}
where we have set $u =w \setminus \{M+1\}$.
The coefficient of the product of $(x_{M+1})_b$  and the monomial \eqref{c mono} in $F(\sum_{i=1}^{M+1}, M, a+1)$ is equal to the product of the expression at line \eqref{alpha factor} and 
\begin{equation} \label{right factor}
\frac{(M+ \sum_{i \geq 2} \nu_i)!
}{ a!b!\prod_{i\geq 2}\nu_i!}
 {\sum_{i=1}^M x_i \choose M-a-b-\sum_{i\geq 2} (i-1)\nu_i}.  
 \end{equation}
 Therefore it is sufficient to prove that sum at lines \eqref{uc factor} and \eqref{u factor} is equal to expression \eqref{right factor}. We prove this next by comparing the coefficients of the variables $x_i$. 
 
 Fix integers $l$ and $n_i \geq 0$ such that $l\leq M$. Consider the monomial 
 \begin{equation} \label{x mono}
 \prod_{i=1}^l x_i^{n_i}.
 \end{equation}
 The coefficient of monomial \eqref{x mono} in the sum at lines \eqref{uc factor} and \eqref{u factor} is 
\begin{align}
& \frac{1}{(a-1)!b!}\sum_{v \subset [1,l]}\sum_{j=0}^M \sum_{\mu \leq_1 \nu, \mu\in C(j)}  \label{summu}\\
& \frac{{  M-j-1-(a-1)-\sum_{i\geq 2} (i-1)\mu_i' \brack \sum_{i \in v^c} n_i}}{ (M-j-1-(a-1)-\sum_{i\geq 2} (i-1)\mu_i' )!} \left( \frac{(\sum_{i\in v^c} n_i)!}{\prod_{i \in v^c} n_i!}\right)\frac{(M-j-1+ \sum_{i \geq 2} \mu_i')! }{ \prod_{i\geq 2}\mu_i'!} \\ 
& \times \frac{{ j-b-\sum_{i\geq 2} (i-1)\mu_i \brack \sum_{i \in v} n_i}}{ (j-b-\sum_{i\geq 2} (i-1)\mu_i )!} \left( \frac{(\sum_{i\in v} n_i)!}{\prod_{i \in v} n_i!}\right)\frac{(j+\sum_{i \geq 2} \mu_i)! }{ \prod_{i\geq 2}\mu_i!} \\ 
& \times{M-l \choose j - |v|}
\end{align}
where we have set $j = |u|$ and used the fact that there are $\displaystyle  {M-l \choose j - |v|}$ ways to choose a set $u \in [1,M]$ such that $v \subset u$ and $v^c  \subset u^c$. 

Now each term in the above sum is independent of $\nu_1, \mu_1$, and $\mu_1'$. We thus re-arrange the summation at line \eqref{summu} to be 
\begin{equation} \label{sum re}
\sum_{v \subset [1,l]}\sum_{\tilde{\mu }\leq \tilde{\nu}} \sum_{j=0}^M  
\end{equation}
where $\tilde{\nu}$ is the fixed sequence $(\nu_i)_{i\geq2}$; $\tilde{\mu}$ is any sequence of non-negative integers $(\tilde{\mu}_i)_{i \geq 2}$ with 
\[
\tilde{\mu}_i \leq \nu_i \text{ for } i \geq 2
\]
and 
\[
\tilde{\mu}_i' = \nu_i - \tilde{\mu}_i \text{ for } i \geq 2.
\]

The coefficient of term \eqref{x mono} in expression \eqref{right factor} is 
\begin{equation} \label{right mono}
 \frac{(M+ \sum_{i \geq 2} \nu_i)!}{ a!b!\prod_{i\geq 2}\nu_i!}
 \frac{{M-a-b-\sum_{i\geq 2} (i-1)\nu_i \brack \sum_{i=1}^l n_i}}{(M-a-b-\sum_{i\geq 2} (i-1)\nu_i )!} \left( \frac{(\sum_{i=1}^l n_i)!}{\prod_{i =1}^l n_i!}\right). 
\end{equation}
We equate expression \eqref{right mono} with the sum beginning at line \eqref{summu} (using the ordering \eqref{sum re}) and simplify to obtain 
\begin{align}
& \sum_{v \subset [1,l]}\sum_{\tilde{\mu }\leq \tilde{\nu}} \sum_{j=0}^M  \\
& a\frac{(M-j-1+ \sum_{i \geq 2} \tilde{\mu}_i')! }{ (M-j-1-(|v^c|-1))!}\frac{{  M-j-1-(a-1)-\sum_{i\geq 2} (i-1)\tilde{\mu}_i' \brack \sum_{i \in v^c} n_i}}{ (M-j-1-(a-1)-\sum_{i\geq 2} (i-1)\tilde{\mu}_i' )!} (\sum_{i\in v^c} n_i)! \\ 
& \times \frac{(j+\sum_{i \geq 2} \tilde{\mu}_i)! }{ (j-|v|)!} \frac{{ j-b-\sum_{i\geq 2} (i-1)\mu_i \brack \sum_{i \in v} n_i}}{ (j-b-\sum_{i\geq 2} (i-1)\tilde{\mu}_i )!} (\sum_{i\in v} n_i)!\\ 
& \times \prod_{i\geq 2} \frac{(\tilde{\mu}_i+\tilde{\mu}_i' )!}{(\tilde{\mu}_i)!(\tilde{\mu}_i')!}\\
&= \frac{(M+ \sum_{i \geq 2} \nu_i)!}{ (M-l)!}
 \frac{{M-a-b-\sum_{i\geq 2} (i-1)\nu_i \brack \sum_{i=1}^l n_i}}{(M-a-b-\sum_{i\geq 2} (i-1)\nu_i )!} (\sum_{i=1}^l n_i)!.
\end{align}
We claim the above equation is true for any integer $M\geq l$. Thus multiply both sides above the above equation by $t^{M-l}$ and sum over $M \geq l$. The claim is thus equivalent to the equality of power series 
\begin{align}
& \sum_{v \subset [1,l]}\sum_{\tilde{\mu }\leq \tilde{\nu}} \\
& (\frac{d}{dt})^{\sum_{i\geq2} \tilde{\mu}_i'+|v^c|-1}(a t^{a-1} (-\ln(1-t))^{\sum_{i\in v^c}n_i} \prod_{i \geq 2} (t^i)^{\tilde{\mu}_i'} )\\ 
&\times(\frac{d}{dt})^{\sum_{i\geq2} \tilde{\mu}_i+|v|}(t^b (-\ln(1-t))^{\sum_{i\in v}n_i} \prod_{i \geq 2} (t^i)^{\tilde{\mu}_i} )\\ 
& \times \prod_{i\geq 2} \frac{(\tilde{\mu}_i+\tilde{\mu}_i' )!}{(\tilde{\mu}_i)!(\tilde{\mu}_i')!}\label{mu count}\\
&= (\frac{d}{dt})^{\sum_{i\geq2} \nu_i+l}  (t^a t^b (-\ln(1-t))^{\sum_{i=1}^l n_i} \prod_{i \geq 2} (t^i)^{\nu_i} ).
\end{align}
This equality follows from Theorem \ref{t deriv set} using the ring $R$ of power series in $t$; $\delta$ differentiation with respect to $t$; and
\begin{align*}
&f_A = t^a\\ 
&f_B = t^b \\ 
&f_i = (-\ln(1-t))^{n_i} \text{ for } 1 \leq i \leq l,
\end{align*} 
such that for each integer $h\geq 2$ there are exactly $\nu_h$ indices $i$ with $l+1\leq i \leq l+\sum_{\geq 2} \nu_h$ and  
\[
f_i = t^h.
\]
Here we have used the fact that the number \eqref{mu count} counts the number of subsets of $[1,l+\sum_{\geq 2} \nu_h]$ that contain $v$, do not contain $v^c$, and contain exactly $\tilde{\mu}_h$ numbers $i$ such that $f_i = t^h$, for each $h\geq2$. This proves the claim and completes the proof.
\end{proof}

\begin{theorem} \label{t Taylor int}
Suppose $\gamma_i \in \mathbb{Z}$ for $1 \leq i \leq d$. Then
\begin{equation} \label{Taylor int}
\frac{\partial(\phi,\vect{n})}{\prod_{i=1}^d n_i!} \in \alpha^{\vect{n} \cdot \vect{\gamma} }\mathbb{Z}[\alpha^{-1 }, c_1^{-1}, c_2, c_3, \ldots, c_{\Sigma \vect{n}}].
\end{equation}
\end{theorem} 
\begin{proof} 
We use induction on $\Sigma \vect{n}$. From Theorem \ref{t phi F}, statement \eqref{Taylor int} is true when $\Sigma \vect{n}=1$. Assume it is true for all $\vect{n}$ with $\Sigma \vect{n}\leq m$ for some $m \geq 1$. Given a $\vect{n}$ with $\Sigma \vect{n}=m+1$, fix an $r$ such that $n_r\geq 1$. Consider
\begin{equation} \label{partial n a}
\partial_{\vect{n}}(a_r \phi^{\gamma_r}).
\end{equation}
After setting $\vect{a}$ to $\vect{0}$, expression \eqref{partial n a} is equal to 
\begin{equation} \label{after 0}
n_r \sum_{k=1}^{\Sigma \vect{n}-1} (\gamma_r)_k \alpha^{\gamma_r - k} \sum_{J \in \mathrm{Parts}(I,k)} \prod_{i=1}^k \partial(\phi,J_i)
\end{equation}
where $I$ is the ordered multiset satisfying $I(h) \leq I(h+1)$, and with $r$ occurring with multiplicity $n_r-1$ and $i$ occurring with multiplicity $n_i$ for all other $i$. Now fix a $k$ and $J$. 
Suppose that there are $v$ distinct sets out of the $k$ sets of $J$. Call these sets $\tilde{J}_1, \ldots, \tilde{J}_v$ with $\tilde{J}_i$ appearing $b_i$ times in $J$. The number of $J'\in \mathrm{Parts}(I,k)$ equivalent to $J$ is then 
\[
\frac{1}{\prod_{i=1}^v b_i!}\prod_{i=1}^d \frac{\#(i,J)!}{\prod_{l=1}^k \#(i,J_l)!}.
\]
Collect these terms for $J'$ equivalent to $J$ and divide by $\prod_{i=1}^d n_i!$. The total coefficient 
 is then 
\begin{align*}
&\frac{n_r (\gamma_r)_k \alpha^{\gamma_r - k}}{n_r\prod_{i=1}^d \#(i,J)!}   \frac{1}{\prod_{i=1}^v b_i!}\prod_{i=1}^d \frac{\#(i,J)!}{\prod_{l=1}^k \#(i,J_l)!}\prod_{i=1}^k  \partial(\phi,J_i)  \\
=& {\gamma_r \choose k} \mathrm{multinomial}((b_i)_{i=1}^v)\prod_{l=1}^k \frac{\partial(\phi,J_l)}{\prod_{i=1}^d \#(i,J_l)!}
\end{align*}   
where have used 
\[
\sum_{i=1}^v b_i = k.
\]
By the induction hypothesis 
\[
 \frac{\partial(\phi,J_l)}{\prod_{i=1}^d \#(i,J_l)!} \in \displaystyle \alpha^{\sum_{m\in J_l}\gamma_m} \mathbb{Z} [\alpha^{-1}, c_1^{-1}, c_2, \ldots, c_{|J_l|}],
\]
and $\displaystyle {\gamma_r \choose k}$ is an integer by the assumption that $\gamma_r$ is an integer. 

Next, consider a fixed $k$-th term at line \eqref{k c} where $I$ is the ordered multiset satisfying $I(h) \leq I(h+1)$, and with $n_i = \#(i, I)$. Collecting terms for all equivalent $J$ and applying similar reasoning above yields an element in the set at line \eqref{Taylor int}, with $k!$ taking the role of $(\gamma_r)_k$. This completes the proof.
\end{proof}

\section{A transformation rule }

Now we present definitions for Theorem \ref{t phi transform}.

Using base function \eqref{2 term base}, we rename its zero $\alpha$ by $\alpha_1$, and we denote the corresponding zero $\phi(\vect{a}; \vect{\gamma}, g,\alpha_1)$ of $f(z)$ by 
\[
\phi(\vect{a};  \vect{\gamma}, b,\beta_1).
\]

For a non-zero $\beta_2 \in \mathbb{C}$, we denote $\alpha_2$ by
\begin{equation} \label{alpha2}
 \alpha_2=\alpha_1^{\frac{1}{\beta_2}}. 
\end{equation}

Let $M(d)$ denote the set of $d$-tuples $\vect{n}$
\[
\vect{n} = (n_i)_{i=1}^d
\]
 of non-negative integers $n_i$ such that $\Sigma \vect{n} \geq 1$. 
 
 For $\vect{n}\in M(d)$, let $V(k, \vect{n})$ denote the set of $k \times d$ arrays $\vect{v}$ 
 \[
 \vect{v} =  (v_{i,j}) \text{ for } 1 \leq i \leq d \text{ and } 1 \leq j \leq k
 \]
 where 
\[
\sum_{j=1}^k v_{i,j} = n_i
\]
and $v_j \in M(d)$ where 
\[
v_j=(v_{i,j})_{i=1}^d.
\]
Let $\mathrm{perm}(\vect{v})$ denote the number of permutations of the multiset $(v_j)_{j=1}^k$.
Let $w(\vect{n}, \vect{\gamma}, \beta)$ denote
 \[
w(\vect{n}, \vect{\gamma}, \beta)= ( \beta_1^{-1}-1+\beta_1^{-1}\sum_{i =1}^d n_i\gamma_{i} )_{\Sigma \vect{n}-1}.
 \]
 
\begin{theorem}\label{t phi transform} With the above notation, as Taylor series about $\vect{a}=\vect{0}$ 
\begin{align*}
\phi(\vect{a};  \vect{\gamma}, b,\beta_1)^{\frac{1}{\beta_2}} &= \phi(\vect{a};  \beta_2 \vect{\gamma}, b,\beta_2 \beta_1).
\end{align*}
\end{theorem}
\begin{proof}
We have by Theorem \ref{t main} 
\begin{align*}
\phi(\vect{a};  \vect{\gamma}, b,\beta_1) &= \alpha_1\left(1+\sum_{\vect{n} \in M(d)} (-1)^{\Sigma \vect{n}}\frac{\alpha_1^{ \sum_{i=1}^d n_i(\gamma_{i}-1)}}{g'(\alpha)^{\Sigma \vect{n}} } \beta_1^{\Sigma \vect{n}-1}( \beta_1^{-1}-1+\beta_1^{-1}\sum_{i =1}^d n_i\gamma_{i} )_{\Sigma \vect{n}-1} \prod_{i=1}^d \frac{a_i^{n_i}}{n_i!}\right)\\
&=\alpha_1\left(1+\beta_1^{-1} \sum_{\vect{n} \in M(d)} 
\alpha_1^{ \sum_{i=1}^d n_i\gamma_{i}}  w(\vect{n}, \vect{\gamma}, \beta)\prod_{i=1}^d \frac{a_i^{n_i}}{n_i!}\right).
\end{align*}
Therefore
\begin{align*}
\phi(\vect{a};  \vect{\gamma}, b,\beta_1)^{\frac{1}{\beta_2}} &= \alpha_1^{\frac{1}{\beta_2}}\left(1+\beta_1^{-1} \sum_{\vect{n}\in M(d)} 
\alpha_1^{ \sum_{i=1}^d n_i\gamma_{i}}  w(\vect{n}, \vect{\gamma}, \beta) \prod_{i=1}^d \frac{a_i^{n_i}}{n_i!}\right)^{\frac{1}{\beta_2}}\\ 
&= \alpha_2\left(1+ \sum_{k=1}^\infty \beta_1^{-k}{\beta_2^{-1} \choose k}( \sum_{\vect{n}\in M(d)} 
\alpha_1^{ \sum_{i=1}^d n_i\gamma_{i}}  w(\vect{n}, \vect{\gamma}, \beta)\prod_{i=1}^d \frac{a_i^{n_i}}{n_i!})^k\right)
\end{align*}
by the binomial theorem, using the fact that $\vect{a}$ is sufficiently close to $\vect{0}$.

For an $\vect{m} \in M(d)$, the coefficient of $\displaystyle  \prod_{i=1}^d \frac{a_i^{m_i}}{m_i!}$ in the above sum is 
\begin{equation} \label{m coefficient}
\alpha_2  \frac{\alpha_1^{\sum_{i=1}^d m_i \gamma_i} }{\prod_{i=1}^d m_i!} \sum_{k=1}^{\Sigma \vect{m}}\beta_1^{-k} {\beta_2^{-1} \choose k} \sum_{v \in V(k,\vect{m})}  (\mathrm{perm}(v) \prod_{i=1}^d \mathrm{mult}((v_{i,j})_{j=1}^k))\prod_{j=1}^k w(\vect{v_j}, \vect{\gamma}, \beta)
\end{equation}
Let $I \in \mathrm{Multiset}(d)$ be the ordered multiset in which $\#(i,I)=m_i$ and ordered so that $I(j) \leq I(j+1)$. For $J \in \mathrm{Parts}(I,k)$, let $w(J, \vect{\gamma}, \beta)$ denote
 \[
w(J, \vect{\gamma}, \beta)= \prod_{h=1}^k( \beta^{-1}-1+\beta^{-1}\sum_{i \in J_h}^d \gamma_{i} )_{|J_h|-1}.
 \]
Then we may write expression \eqref{m coefficient} as 
\begin{equation} \label{may write m coeff}
\alpha_2  \frac{\alpha_1^{\sum_{i=1}^d m_i \gamma_i} }{\beta_1\beta_2\prod_{i=1}^d m_i!} \sum_{k=1}^{\Sigma \vect{m}} (\beta_2^{-1}-1)_{k-1} (\beta^{-1})^{k-1}\sum_{J \in \mathrm{Parts}(I,k)}  w(J, \vect{\gamma}, \beta).
\end{equation}
By Theorem \ref{t nu}, we have 
\[
(\beta^{-1})^{k-1}\sum_{J \in \mathrm{Parts}(I,k)}  w(J, \vect{\gamma}, \beta)= \frac{1}{(k-1)!}\sum_{r=0}^{k-1} (-1)^{k-1-r}{k-1 \choose r} ((r+1)\beta_1^{-1}-1+\sum_{i\in I} \gamma_i)_{|I|-1}
\]
Let $c_k$ denote the right side of the above equation. By Lemma \ref{Newton series}, we have expression \eqref{may write m coeff} is equal to 
\[
\alpha_2  \frac{\alpha_1^{\sum_{i=1}^d m_i \gamma_i} }{\beta_1\beta_2\prod_{i=1}^d m_i!}  ((\beta_1\beta_2)^{-1}-1+\beta_1^{-1}\sum_{i=1}^d m_i \gamma_i)_{\Sigma \vect{m}-1}.
\]
Using equation \eqref{alpha2}, we see that this is the coefficient of $\displaystyle  \prod_{i=1}^d \frac{a_i^{m_i}}{m_i!}$ in the series for $\phi(\vect{a};  \beta_2 \vect{\gamma}, b,\beta_2 \beta_1)$. This completes the proof. 
\end{proof}

\section{Further work} \label{further}
\begin{itemize}

\item Find a combinatorial meaning of the integer coefficients of the $c_i$ terms using trees. 
 
\item Find a combinatorial meaning of the factored coefficients in \cite{DeFranco} when 
\[
g(z) = 1+\frac{a_{i+d}}{a_i} z^{d}.
\]
 
\item Generalize NRS using these combinatorial meanings. 

\item Directly prove the factorization of coefficients in \cite{DeFranco} from the formulas here. 

\item Generalize Theorem \ref{t phi transform} for complex-exponent polynomials $g(z)$ with multiple terms. 

\item Use the method of Lagrange or a modification to construct iterated radical formulas for polynomial zeros and relate them to Taylor series.

\end{itemize}

\end{document}